\documentclass[10pt]{article}

\usepackage{latexsym,enumerate}
\usepackage{amsmath,amsthm,amsopn,amstext,amscd,amsfonts,amssymb}
\usepackage[ansinew]{inputenc}
\usepackage{verbatim}
\usepackage{graphicx}
\usepackage{dsfont}

\usepackage{multicol}
\usepackage{color}

\newtheorem{theorem}{Theorem}[section]
\newtheorem{lemma}[theorem]{Lemma}
\newtheorem{proposition}[theorem]{Proposition}
\newtheorem{corollary}[theorem]{Corollary}

\newtheorem{remark}[theorem]{Remark}

\DeclareMathOperator{\diam}{diam}

\DeclareMathOperator{\inte}{int}

\newcommand{\spb}[1]{\smallskip}
\newcommand{\mpb}[1]{\medskip}
\newcommand{\bpb}[1]{\bigskip}

\renewcommand{\a}{\alpha}
\renewcommand{\b}{\beta}
\newcommand{\e}{\varepsilon}
\renewcommand{\d}{\delta}
\newcommand{\g}{\gamma}
\newcommand{\G}{\Gamma}

\renewcommand{\L}{\mathcal{L}}

\title{New inequalities on the hyperbolicity constant of line graphs}

\author{Walter Carballosa$^{1}$, Jos\'e M. Rodr{\'\i}guez$^{1}$ and Jos\'e M. Sigarreta$^{2}$
\\
\\
$^1${\small Departamento de Matem\'aticas}
\\ {\small Universidad Carlos III de Madrid, Av. de la Universidad 30,
28911 Legan\'es, Madrid, Spain}\\
{\small waltercarb\@@gmail.com, jomaro\@@math.uc3m.es}
\\
$^2${\small Facultad de Matem\'aticas}
\\{\small Universidad Aut\'onoma de Guerrero,  Carlos E. Adame 5,
Col. La Garita, Acapulco, Guerrero, M\'{e}xico.} \\
{\small jsmathguerrero\@@gmail.com}
}

\begin{document}


\begin{center}
{\LARGE New inequalities on the hyperbolicity constant of line graphs}\\[12pt]
{\large Walter Carballosa$^{1}$, Jos\'e M. Rodr{\'\i}guez$^{1}$, \\[5pt] and Jos\'e M. Sigarreta$^{2}$}\\[12pt]
$^1${\small Departamento de Matem\'aticas}\\
{\small Universidad Carlos III de Madrid,\\Av. de la Universidad 30,
28911 Legan\'es, Madrid, Spain}\\
{\small waltercarb\@@gmail.com, jomaro\@@math.uc3m.es}\\[3pt]

$^2${\small Facultad de Matem\'aticas}
\\{\small Universidad Aut\'onoma de Guerrero, \\ Carlos E. Adame 5,
Col. La Garita, Acapulco, Guerrero, M\'{e}xico.} \\
{\small jsmathguerrero\@@gmail.com}
\end{center}

\begin{abstract}
If X is a geodesic metric space and $x_1,x_2,x_3\in X$, a {\it
geodesic triangle} $T=\{x_1,x_2,x_3\}$ is the union of the three
geodesics $[x_1x_2]$, $[x_2x_3]$ and $[x_3x_1]$ in $X$. The space
$X$ is $\d$-\emph{hyperbolic} $($in the Gromov sense$)$ if any side
of $T$ is contained in a $\d$-neighborhood of the union of the two
other sides, for every geodesic triangle $T$ in $X$. We denote by
$\d(X)$ the sharp hyperbolicity constant of $X$, i.e.
$\d(X):=\inf\{\d\ge 0: \, X \, \text{ is $\d$-hyperbolic}\,\}\,. $
The main result of this paper is the inequality
$\d(G) \le \d(\L(G))$
for the line graph $\L(G)$ of every graph $G$.
We prove also the upper bound
$\d(\L(G)) \le 5 \d(G)+ 3 l_{max}$, where $l_{max}$ is the supremum of the lengths of
the edges of $G$.
Furthermore, if every edge of $G$ has length $k$, we obtain $\d(G) \le \d(\L(G)) \le 5 \d(G)+ 5k/2$.
\end{abstract}



\section{Introduction.}

Hyperbolic spaces play an important role in the geometric
group theory and in the geometry of negatively curved
spaces (see \cite{ABCD, GH, G1}).
The concept of Gromov hyperbolicity grasps the essence of negatively curved
spaces like the classical hyperbolic space, Riemannian manifolds of
negative sectional curvature bounded away from $0$, and of discrete spaces like trees
and the Cayley graphs of many finitely generated groups. It is remarkable
that a simple concept leads to such a rich
general theory (see \cite{ABCD, GH, G1}).

The study of mathematical properties of Gromov hyperbolic spaces
and its applications is a topic of recent and increasing interest in
graph theory; see, for instance \cite{BRS,
BRSV2,BRST,BHB1,CPRS,CRSV,K50,
K27,K21,K22,K23,K24,K56,MRSV,MRSV2,PeRSV,PRST,PRSV,PRT1,RSVV,
RT1,RT3,T,WZ}.

The theory of Gromov spaces was used initially for the study of
finitely generated groups (see \cite{G1}
and the references therein),
where it was demonstrated to have practical importance. This theory was applied principally
to the study of automatic groups (see \cite{O}), which play a
role in the science of computation.
The concept of hyperbolicity appears also in discrete mathematics, algorithms
and networking. For example, it has been shown empirically
in \cite{ShTa} that the internet topology embeds with better accuracy
into a hyperbolic space than into a Euclidean space
of comparable dimension. A few algorithmic problems in
hyperbolic spaces and hyperbolic graphs have been considered
in recent papers (see \cite{ChEs, Epp, GaLy, Kra}).
Another important
application of these spaces is secure transmission of information by
internet (see \cite{K27,K21,K22}). In particular, the hyperbolicity
plays an important role in the spread of viruses through the
network (see \cite{K21,K22}).
The hyperbolicity is also useful in the study of DNA data (see \cite{BHB1}).

In recent years several researchers have been interested in showing
that metrics used in geometric function theory are Gromov
hyperbolic. For instance, the Gehring-Osgood $j$-metric
is Gromov hyperbolic; and the Vuorinen $j$-metric is not Gromov
hyperbolic except in the punctured space (see \cite{Ha}).
The study of Gromov hyperbolicity of the quasihyperbolic and the
Poincar\'e metrics is the subject of
\cite{
BB,BHK,HLPRT,HPRT,
PRT1,PRT2,PT,RT1,
RT3}.
In particular, in \cite{PRT1,RT1,RT3,T} it is proved the equivalence of
the hyperbolicity of many negatively curved surfaces and the hyperbolicity of a simple graph; hence, it is
useful to know hyperbolicity criteria for graphs.

In our study on hyperbolic graphs we use the notations of \cite{GH}. Let $(X,d)$ be a metric space and let $\g:[a,b]\longrightarrow X$ be a continuous function.
We say that $\g$ is a
\emph{geodesic} if $L(\g|_{[t,s]})=d(\g(t),\g(s))=|t-s|$ for every $s,t\in [a,b]$,
where $L$ denotes the length of a curve. We
say that $X$ is a \emph{geodesic metric space} if for every $x,y\in
X$ there exists a geodesic joining $x$ and $y$; we denote by $[xy]$
any of such geodesics (since we do not require uniqueness of
geodesics, this notation is ambiguous, but it is convenient). It is
clear that every geodesic metric space is path-connected.
If the metric space $X$ is
a graph, we use the notation $[u,v]$ for the edge joining the vertices $u$ and $v$.

In order to consider a graph $G$ as a geodesic metric space, we must identify
any edge $[u,v]\in E(G)$ with the real interval $[0,l]$ (if $l:=L([u,v])$);
therefore, any point in the interior
of any edge is a point of $G$ and, if we consider the edge $[u,v]$ as a graph with just one edge,
then it is isometric to $[0,l]$.
A connected graph $G$ is naturally equipped with a distance
defined on its points, induced by taking shortest paths in $G$.
Then, we see $G$ as a metric graph.

Throughout the paper we just consider simple (without loops and multiple edges) connected and
locally finite (i.e., in each ball there are just a finite number of edges) graphs;
these properties guarantee that the graphs are geodesic metric spaces.
Note that the edges can have arbitrary lengths.
We want to remark that by \cite[Theorems 8 and 10]{BRSV2} the study of the hyperbolicity of graphs with loops and multiple edges can be reduced
to the study of the hyperbolicity of simple graphs.

If $X$ is a geodesic metric space and $J=\{J_1,J_2,\dots,J_n\}$
is a polygon, with sides $J_j\subseteq X$, we say that $J$ is $\d$-{\it thin} if for
every $x\in J_i$ we have that $d(x,\cup_{j\neq i}J_{j})\le \d$. We
denote by $\d(J)$ the sharp thin constant of $J$, i.e.,
$\d(J):=\inf\{\d\ge 0: \, J \, \text{ is $\d$-thin}\,\}\,. $ If
$x_1,x_2,x_3\in X$, a {\it geodesic triangle} $T=\{x_1,x_2,x_3\}$ is
the union of the three geodesics $[x_1x_2]$, $[x_2x_3]$ and
$[x_3x_1]$; sometimes we write the geodesic triangle $T$ as
$T=\{[x_1x_2], [x_2x_3], [x_3x_1]\}$.
The space $X$ is $\d$-\emph{hyperbolic} $($or satisfies
the {\it Rips condition} with constant $\d)$ if every geodesic
triangle in $X$ is $\d$-thin. We denote by $\d(X)$ the sharp
hyperbolicity constant of $X$, i.e., $\d(X):=\sup\{\d(T): \, T \,
\text{ is a geodesic triangle in }\,X\,\}.$ We say that $X$ is
\emph{hyperbolic} if $X$ is $\d$-hyperbolic for some $\d \ge 0$. If
$X$ is hyperbolic, then $ \d(X)=\inf\{\d\ge 0: \, X \, \text{ is
$\d$-hyperbolic}\,\}.$
One can check that every geodesic polygon in $X$ with $n$ sides is $(n-2)\d(X)$-thin;
in particular, any geodesic quadrilateral is $2\d(X)$-thin.

There are several definitions of Gromov hyperbolicity.
These different definitions are equivalent in the sense
that if $X$ is $\d$-hyperbolic with respect to the definition $A$,
then it is $\d'$-hyperbolic with respect to the definition $B$
for some $\d'$ (see, e.g., \cite{BHB,GH}).
We have chosen this definition since
it has a deep geometric meaning (see, e.g., \cite{GH}).

The following are interesting examples of hyperbolic spaces.
The real line $\mathbb{R}$ is $0$-hyperbolic: in fact, any point of a geodesic
triangle in the real line belongs to two sides of the triangle
simultaneously, and therefore we can conclude that $\mathbb{R}$ is
$0$-hyperbolic.
The Euclidean plane $\mathbb{R}^2$ is not hyperbolic: it is clear that
equilateral triangles can be drawn with arbitrarily large diameter, so that
$\mathbb{R}^2$ with the Euclidean metric is not
hyperbolic. This argument can be generalized in a similar way to
higher dimensions: a normed vector space $E$ is hyperbolic if and
only if $\dim\ E=1$. Every metric tree with arbitrary length edges is
$0$-hyperbolic: in fact, all points of a geodesic triangle in a tree
belongs simultaneously to two sides of the triangle. Every bounded
metric space $X$ is $(\diam X)/2$-hyperbolic. Every simply connected
complete Riemannian manifold with sectional curvature verifying
$K\leq -c^2$, for some positive constant $c$, is hyperbolic. We
refer to \cite{BHB,GH} for more background and further results.

We want to remark that the main examples of hyperbolic graphs are the trees.
In fact, the hyperbolicity constant of a geodesic metric space can be viewed as a measure of
how ``tree-like'' the space is, since those spaces $X$ with $\delta(X) = 0$ are precisely the metric trees.
This is an interesting subject since, in
many applications, one finds that the borderline between tractable and intractable
cases may be the tree-like degree of the structure to be dealt with
(see, e.g., \cite{CYY}).

Given a Cayley graph (of a presentation with solvable word problem)
there is an algorithm which allows to decide if it is hyperbolic.
However, for a general graph or a general geodesic metric space
deciding whether or not a space is
hyperbolic is usually very difficult.
Therefore, it is interesting to obtain inequalities involving the hyperbolicity
constant.

It is a remarkable fact that the constants appearing in many results in the theory of hyperbolic spaces
depend just on a small number of parameters (also, this is common in the theory of negatively curved surfaces).
Usually, there is no explicit expression for these constants.
Even though sometimes it is possible to estimate the constants, those explicit values obtained are, in general, far from being
sharp (see, e.g., Theorem \ref{invarianza} and \eqref{firstineq} below).

The main result of this paper is the inequality
$\d(G) \le \d(\L(G))$
for the line graph $\L(G)$ of every graph $G$ (see Theorem \ref{t:IneqLineGraph}).

Line graphs were initially introduced in the papers \cite{W} and \cite{Kr},
although the terminology of line graph was used in \cite{HN} for the first time.

There are previous results relating the hyperbolicity constant
of the line graph $\L(G)$ with the hyperbolicity constant of the graph $G$.
In \cite[Theorem 2.4]{CRSV} the authors obtain the inequalities
\begin{equation}
\label{firstineq}
\frac1{12}\, \d(G) - \frac{3}{4} \le
\d(\L(G)) \le 12\, \d(G)+ 18,
\end{equation}
for graphs $G$ with edges of length $1$.
This result allows to obtain the main qualitative result of \cite{CRSV}: the line graph of $G$ is hyperbolic
if and only if $G$ is hyperbolic.
Although the multiplicative and additive constants appearing in \eqref{firstineq}
allow to prove this main result, it is a natural problem to improve the inequalities in \eqref{firstineq}.
In this paper we also improve the second inequality; in fact, Theorem \ref{t:IneqLineGraph} states
\begin{equation}
\label{secondineq}
\d(G) \le \d(\L(G)) \le 5\, \d(G)+ 3 \sup_{e\in E(G)} L(e),
\end{equation}
where here the edges of $G$ can have arbitrary lengths.
The second inequality in \eqref{secondineq}
can be improved for graphs with edges of length $k$ (see Corollary \ref{c:IneqLineGraphk})
in the following way:

\[
\d(G)  \leq  \d(\L(G))  \leq  5 \d(G) + 5k/2.
\]

Also, we obtain for graphs with edges of length $k$ other inequalities involving the hyperbolicity constant of $\L(G)$
(see Theorem \ref{t:j} and Corollary \ref{c:sumdd}).

\

\section{Background and previous results.}
Let $G$ be a graph such that its edges $E(G)=\{e_i\}_{i\in\mathcal{I}}$ have arbitrary lengths.
The \emph{line graph} $\L(G)$ of $G$ is a graph which has a vertex $V_{e_i}\in V(\L(G))$
for each edge $e_i$ of $G$, and an edge joining $V_{e_i}$ and $V_{e_j}$ when $e_i \cap e_j \neq \varnothing$.
Note that we have a complete subgraph $K_n$ in $\L(G)$ corresponding to one vertex $v$ of $G$ with degree $\deg_{G}(v)=n$.
Some authors define the edges of line graph with length $1$ or another fixed constant, but we define the length of the edge $[V_{e_i},V_{e_j}]\in E(\L(G))$ as $(L(e_i) + L(e_j))/2$.
Note that if every edge in $G$ has length $k$, then every edge in $\L(G)$ also has length $k$.

\

Let $(X,d_X)$ and $(Y,d_Y)$  be two metric spaces. A map $f: X\longrightarrow Y$ is said to be
an $(\alpha, \beta)$-\emph{quasi-isometric embedding}, with constants $\alpha\geq 1,\
\beta\geq 0$ if we have for every $x, y\in X$:
$$
\alpha^{-1}d_X(x,y)-\beta\leq d_Y(f(x),f(y))\leq \alpha d_X(x,y)+\beta.
$$
We say that $f$ is $\varepsilon$-\emph{full} if
for each $y \in Y$ there exists $x\in X$ with $d_Y(f(x),y)\leq \varepsilon$.

\medskip

A map $f: X\longrightarrow Y$ is said to be
a \emph{quasi-isometry} if there exist constants $\alpha\geq 1,\
\beta,\varepsilon \geq 0$ such that $f$ is a $\varepsilon$-full
$(\alpha, \beta)$-quasi-isometric embedding.

\medskip

Two metric spaces $X$ and $Y$ are \emph{quasi-isometric} if there exist
a quasi-isometry $f:X\longrightarrow Y$.

\medskip

A fundamental property of hyperbolic spaces is the following:

\begin{theorem}[Invariance of hyperbolicity]\label{invarianza}
Let $(X,d_X)$, $(Y,d_Y)$ be two geodesic metric spaces and
$f: X\longrightarrow Y$ an $(\alpha, \beta)$-quasi-isometry embedding.
\begin{itemize}
\item [i)] If $Y$ is $\d$-hyperbolic, then $X$ is $\d'$-hyperbolic, where $\d'$ is a
constant which just depends on $\d$, $\a$ and $\b$.
\item[ii)] If $f$ is $\e$-full, then $X$ is hyperbolic if and only if $Y$ is hyperbolic.
Furthermore, if $X$ is $\d'$-hyperbolic, then $Y$ is $\d$-hyperbolic, where $\d$ is a
constant which just depends on $\d'$, $\a$, $\b$ and $\e$.
\end{itemize}
\end{theorem}

\medskip

We will need the following result (see \cite[Lemma 2.1]{RT1}):
\begin{lemma}\label{l:ClosedCurve}
Let us consider a geodesic metric space $X$. If every geodesic triangle in $X$ which
is a simple closed curve, is $\d$-thin, then $X$ is $\d$-thin.
\end{lemma}

This lemma has the following direct consequence.
As usual, by \emph{cycle} we mean a simple closed curve, i.e., a path with different vertices in a graph,
except for the last one, which is equal to the first vertex.

\begin{corollary} \label{c:ClosedCurve}
In any graph $G$,
$$
\d(G) = \sup \{
\d(T ) : T \text{ is a geodesic triangle that is a cycle}
\}.
$$
\end{corollary}

\medskip

The next result follows from $\d(X)\le (\diam X)/2$ (see \cite[Theorem 11]{RSVV} for a detailed proof).

\begin{theorem}
\label{t:examplesdelta}
The cycle graphs with every edge of length $1$ verify $\d(C_n)=n/4$ for every $n\ge 3$.
\end{theorem}

This theorem has the following direct consequence.

\begin{corollary}
\label{c:examplesdelta}
Any cycle graph $C$ verifies $\d(C)=L(C)/4$.
\end{corollary}

In this work, $PMV(G)$ will denote the set of points of the graph $G$ which are either vertices or midpoints of its edges.

\smallskip

We will use the following result (see \cite[Theorem 2.7]{BRS}).

\begin{theorem}\label{t:DiscreCycle}
For any hyperbolic graph $G$ with edges of length $k$, there exists a geodesic triangle $T = \{x,y,z\}$ that is a cycle with $\d(T)=\d(G)$ and $x,y,z \in PMV(G)$.
\end{theorem}

\

\section{Inequalities involving the hyperbolicity constant of line graphs.}
We obtain in this section the results on the hyperbolicity constant of a line graph with edges of arbitrary lengths.
The main result in this section is Theorem \ref{t:IneqLineGraph}, which states

\[
\d(G) \le \d(\L(G)) \le 5 \d(G) + 3 l_{max},
\]
with $l_{max}=\sup_{e\in E(G)} L(e)$.

For the sake of clarity and readability, we have opted to state and prove several preliminary lemmas.
This makes the proof of Theorem \ref{t:IneqLineGraph} much more understandable.

Let us consider $Pm(e)$ the midpoint of $e \in E(G)$; also, we denote by $PM(G)$ the set of the midpoints of the edges of $G$, i.e., $PM(G):=\{Pm(e) / e\in E(G)\}$. Besides, let us consider $Pm_\L([V_{e_i},V_{e_j}])$ the point in $[V_{e_i},V_{e_j}] \in E(\L(G))$ with $L( [V_{e_i}Pm_\L([V_{e_i},V_{e_j}])] ) = L(e_i)/2$ (and then $L( [Pm_\L([V_{e_i},V_{e_j}])V_{e_j}] ) = L(e_j)/2$). Analogously, we denote $PM_\L(\L(G))$ the set of these points in each edge of $\L(G)$, i.e., $PM_\L(\L(G)):=\{Pm_\L(e) / e$ $\in E(\L(G))\}$. Note that $Pm_\L([V_{e_i},V_{e_j}])$ is the midpoint of $[V_{e_i},V_{e_j}]$ when $L(e_i)=L(e_j)$; thus, if every edge of $G$ has the same length then $PM_\L(\L(G))$ is the set of midpoints of the edges of $\L(G)$.

Let us consider the sets $PMV(G) := PM(G) \cup V(G)$ and $PM_\L V(\L(G))$ $:= PM_\L(\L(G)) \cup V(\L(G))$.

We define a function $h : PM_\L V(\L(G)) \longrightarrow PMV(G)$ as follows: for every vertex $V_{e}$ of $V(\L(G))$, the image via $h$ of $V_e$ is $Pm(e)$, and for every $Pm_\L([V_{e_i},V_{e_j}])$ in $PM_\L(\L(G))$, the image via $h$ of $Pm_\L ([V_{e_i},V_{e_j}])$ is the vertex $e_i \cap e_j$ in $V(G)$, i.e.,

\begin{equation}\label{eq:QIsometry}
    h(x):=\left\{
            \begin{array}{ll}
              Pm(e), & \hbox{if $x=V_e \in V(\L(G))$,} \\
              e_i \cap e_j, & \hbox{if $x=Pm_\L([V_{e_i},V_{e_j}]) \in PM_\L(\L(G))$.}
            \end{array}
          \right.
\end{equation}

\begin{remark}
If $x\in PM(G)$, then $h^{-1}(x)$ is a single point, but otherwise, $h^{-1}(x)$ can have more than one point.
\end{remark}

   \begin{figure}[h]
     \centering
     \includegraphics[totalheight=5cm]{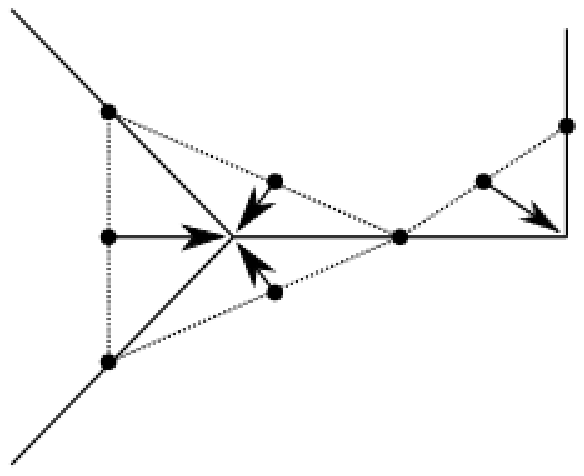}
     \caption{Graphical view of $h$.}
     \label{fig:OutlineH}
   \end{figure}

\medskip

The function $h$ defined in \eqref{eq:QIsometry} can be extended to $\L(G)$.
Note that every point $x_0\in \L(G)\setminus PM_\L V(\L(G))$ is located in $\L(G)$ between one vertex $V_e$ and one point $Pm_\L(V_eV_{e_0})$. For each $x_0\in \inte([V_e Pm_\L([V_eV_{e_0}])])$ we define $h(x_0)$ as the point $x\in \inte[Pm(e) h(Pm_\L([V_eV_{e_0}]))]$ such that $L([xPm(e)]) \ = \ L([x_0V_e])$; \ hence, $\ L([x_0V_e]) \ = \ L([h(x_0)h(V_e)]) \ $ and \newline$L([x_0 Pm_\L(V_eV_{e_0})])=L([h(x_0) h(Pm_\L(V_eV_{e_0}))])$.

In what follows we denote by $h$ this extension.

\smallskip

We call \emph{half-edge} in $G$ a geodesic contained in an edge with an endpoint in $V(G)$ and an endpoint in $PM(G)$;
similarly, a \emph{half-edge} in $\L(G)$ is a geodesic contained in an edge with an endpoint in $V( \L(G) )$ and an endpoint in $PM_\L ( \L(G) )$.

\begin{proposition}\label{p:Lipschitz}
   $h$ is an $1$-Lipschitz continuous function, i.e.,

   \begin{equation}\label{eq:IneqQIso1}
      d_{G}(h(x) , h(y)) \le d_{\L(G)}(x , y) \, , \quad \forall \, x,y \in \L(G).
   \end{equation}
\end{proposition}

\begin{proof}
First of all note that, by definition of $\L(G)$, we have for every $x^{\prime}$ $\in h(\L(G))\cap PMV(G)$,

$$
    |h^{-1}(x^{\prime})|=\left\{
                  \begin{array}{ll}
                    1, & \hbox{if $x^{\prime}\in PM(G)$,} \\
                    \deg_{G}(x^{\prime})(\deg_{G}(x^{\prime})-1)/2, & \hbox{if $x^{\prime}\in V(G)$.}
                  \end{array}
                \right.
$$

In order to prove \eqref{eq:IneqQIso1}, we verify that
\begin{equation}\label{eq:isometric1}
    d_{G}(x^{\prime} , y^{\prime}) = d_{\L(G)}(h^{-1}(x^{\prime}) , h^{-1}(y^{\prime})) \, , \quad \forall \, x^{\prime},y^{\prime} \in h(\L(G))\cap PMV(G).
\end{equation}

We study separately the different cases of $x^{\prime},y^{\prime} \in h(\L(G))\cap PMV(G)$.
\begin{description}
  \item[Case 1] {$x^{\prime},y^{\prime} \in PM(G)$.}

             Let us consider $x^{\prime}:= Pm(e_i)$ and $y^{\prime}:= Pm(e_j)$ with $e_i,e_j \in E(G)$, and define $d := d_{G} (Pm(e_i) ,$ $Pm(e_j)) \geq 0$.

             If $\ d=0$, then $\ e_i=e_j$, \ so, $\, h^{-1}(Pm(e_i)) \ = \ h^{-1}(Pm(e_j)) \, $ and \newline$d_{\L(G)}(h^{-1}(x^{\prime}),h^{-1}(y^{\prime})) = 0$.

             If $d>0$, then $e_i \neq e_j$ and
             $d_{G}(Pm(e_i),Pm(e_j))=(L(e_i)+L(e_j))/2+d_{G}(e_i,e_j)$.
             Note that, if $d_{G}(e_i,e_j)=0$ and $e_i \neq e_j$, then
             $d_G(x^{\prime},y^{\prime})=(L(e_i)+L(e_j))/2=d_{\L(G)}(V_{e_i},V_{e_j})=d_{\L(G)}(h^{-1}(x^{\prime}),h^{-1}(y^{\prime}))$.
             If $d_G(e_i,e_j)>0$, then a geodesic $\g$ joining $e_i$ and $e_j$ in $G$ contains the edges
             $e_{i_1},\ldots,e_{i_{r}}$ in this order, with $r\geq 1$. Now, we have that
             $d_{G}(e_i,e_j)=\sum_{k=1}^{r}L(e_{i_k})$; hence,
             $V_{e_i}V_{e_{i_1}}\ldots V_{e_{i_r}}V_{e_j}$ is a path joining
             $V_{e_i}$ and $V_{e_j}$ with length $d$. So,
             $d_{\L(G)}(h^{-1}(x^{\prime}),h^{-1}(y^{\prime}))$ $\le d$.

             We prove now that $d_{\L(G)} (h^{-1}(x^{\prime}) , h^{-1}(y^{\prime})) = d$. Seeking for a contradiction, assume that $d_{\L(G)}(h^{-1}(x^{\prime}) , h^{-1}(y^{\prime}))=d_{\L(G)}(V_{e_i},V_{e_j})<d$.
             Hence, there exists $V_{e_{j_1}}, \ldots,V_{e_{j_m}}$ such that
             $V_{e_i}V_{e_{j_1}} \ldots$ $V_{e_{j_m}}V_{e_j}$ is a geodesic in
             $\L(G)$ joining $V_{e_i}$ and $V_{e_j}$ with length
             $(L(e_i)+L(e_j))/2+\sum_{k=1}^{m}L(e_{j_k})<d$. Since $d=(L(e_i)+L(e_j))/2+d_{G}(e_i,e_j)$, we have $\sum_{k=1}^{m}L(e_{j_k}) < d_{G} (e_i , e_j)$.
             By definition of $\L(G)$ we have that
             $\g^*:=e_{j_1}\cup\ldots\cup e_{j_m}$ is a path in $G$
             joining $e_i$ and $e_j$ with length
             $\sum_{k=1}^{m}L(e_{j_k})<d_{G}(e_i,e_j)$. This is the contradiction we were looking for; so we
             have $d_{\L(G)}(h^{-1}(x^{\prime}),h^{-1}(y^{\prime}))=d_{G}(x^{\prime},y^{\prime})$.

  \item[Case 2 ] {$x^{\prime} \in PM(G)$ and $y^{\prime}\in V(G)$.}

             Let us consider $x^{\prime}:= Pm(e)$ with $e\in E(G)$ and $y^{\prime} \in V(G)\setminus\{w\in V(G)/\deg_{G}(w)=1\}$, and define $d := d_{G}(e,y^{\prime})$; then $d_{G} (Pm(e) , v)$ $= d + L(e)/2$. Note that if $y^{\prime}\in V(G)$ and $\deg_G(y^{\prime})=1$, then $y^{\prime}\notin h(\L(G))$.

             If $d=0$, then $y$ is an endpoint of $e$ and $d_{\L(G)}(V_e,h^{-1}(y^{\prime})) = L(e)/2$; note that $|h^{-1}(y^{\prime})|=\deg_{G}(y^{\prime})[\deg_{G}(y^{\prime})-1]/2$, where $|A|$ denotes the cardinality of the set $A$.

             If $d_{G}(e,y^{\prime})=d>0$, then there exist $e_{i_1},\ldots,e_{i_r}\in E(G)$ such that $\g:=e_{i_1}\cup\ldots\cup e_{i_r}$ is a geodesic joining $e$ and $y^{\prime}$ in $G$ with length $d=\sum_{k=1}^{r}L(e_{i_k})$. Note that $e,e_{i_1}$ are different and adjacent edges. So, we have that $V_eV_{e_{i_1}}\ldots V_{e_{i_r}}$ is a path in $\L(G)$ joining $V_e$ and $V_{e_{i_r}}$ with length $L(e)/2 + \sum_{k=1}^{r}L(e_{i_k}) - L(e_{i_r})/2$. Since $y^{\prime}$ is an endpoint of $e_{i_r}$, we have $d_{\L(G)}(h^{-1}(y^{\prime}), V_{e_{i_r}})=L(e_{i_r})/2$ and $d_{\L(G)}(h^{-1}(y^{\prime}), V_{e}) \le d + L(e)/2$.

             We prove now that $d_{\L(G)}(h^{-1}(y^{\prime}), V_{e}) = d + L(e)/2$. Seeking for a contradiction, assume that
             $d_{\L(G)}(h^{-1}(y^{\prime}), V_{e}) < d + L(e)/2$.
             Hence, there exists $V_{e_{j_1}},\ldots,V_{e_{j_m}}$ such that
             $V_{e}V_{e_{j_1}}\ldots V_{e_{j_m}}\cup [V_{e_{j_m}}z]$ is a
             geodesic of $\L(G)$ joining $V_{e}$ and $z\in h^{-1}(y^{\prime})$ with
             length $L(e)/2+\sum_{k=1}^{m}L(e_{j_k})< d + L(e)/2$. We have
             $z=Pm_{\L}([V_{e_{j_m}},V_{e_s}])$ with $e_{j_m},e_s$ edges in $G$ starting in $y^{\prime}$.
             By definition of $\L(G)$ we have that
             $\g^*:=e_{j_1}\cup\ldots\cup e_{j_m}$ contains a path in $G$
             joining $e$ and $y^{\prime}$ with length at most
             $\sum_{k=1}^{m} L(e_{j_k})<d$. This is the contradiction we were looking for; so we
             have $d_{\L(G)}(h^{-1}(x^{\prime}),h^{-1}(y^{\prime}))=d_{G}(x^{\prime},y^{\prime})$.

  \item[Case 3] {$x^{\prime},y^{\prime} \in V(G)$.}

             Let us consider $x^{\prime},y^{\prime} \in V(G)\setminus\{v\in V(G)/\deg_{G}(v)=1\}$, and define $d:=d_{G} (x^{\prime} , y^{\prime}) \geq 0$.

             If $d=0$, then $x^{\prime}=y^{\prime}$, so $d_{\L(G)}(h^{-1}(x^{\prime}),h^{-1}(y^{\prime}))=0$.

             If $d_{G}(x^{\prime},y^{\prime})=d>0$, then there exists $e_{i_1},\ldots,e_{i_r}\in E(G)$ such that $\g:=e_{i_1}\cup\ldots\cup e_{i_r}$ is a geodesic joining $x^{\prime}$ and $y^{\prime}$ in $G$ with length $d=\sum_{k=1}^{r}L(e_{i_k})$. So, we have that there exist $a\in h^{-1}(x^{\prime})$ and $b\in h^{-1}(y^{\prime})$ such that $[a V_{e_{i_1}}]\cup V_{e_{i_1}}\ldots V_{e_{i_r}}\cup [ V_{e_{i_r}}b]$ is a path in $\L(G)$ joining $a$ and $b$ with length $\sum_{k=1}^{r} L(e_{i_k}) = d$. Then, we have that $d_{\L(G)}(h^{-1}(x^{\prime}), h^{-1}(y^{\prime})) \le d$.

             We prove now that $d_{\L(G)}(h^{-1}(x^{\prime}), h^{-1}(y^{\prime})) = d$. Seeking for a contradiction, assume that
             $d_{\L(G)}(h^{-1}(x^{\prime}), h^{-1}(y^{\prime})) < d$.
             Hence, there exist $\a\in h^{-1}(x^{\prime})$, $\b\in h^{-1}(y^{\prime})$ and
             $V_{e_{j_1}},\ldots,V_{e_{j_m}}$ vertices in $\L(G)$ such that
             $[\a V_{e_{j_1}}]\cup V_{e_{j_1}}\ldots
             V_{e_{j_m}}\cup[V_{e_{j_m}}\b]$ is a
             geodesic joining $\a$ and $\b$ in $\L(G)$ with
             length $\sum_{k=1}^{m} L(e_{j_k})< d$. We have
             $\a=Pm_\L([V_{e_s^1},V_{e_{j_1}}])$ with $e_{j_1},e_s^1$ edges in $G$ starting in $x^{\prime}$, and
             $\b=Pm_\L([V_{e_{j_m}},V_{e_s^2}])$ with $e_{j_m},e_s^2$ edges in $G$ starting in $y^{\prime}$.
             By definition of $\L(G)$ we have that
             $\g^*:=e_{j_1}\cup\ldots\cup e_{j_m}$ contains a path in $G$ joining $x^{\prime}$ and $y^{\prime}$ with length at most
             $\sum_{k=1}^{m} L(e_{j_k})<d$. This is the contradiction we were looking for; so we
             have $d_{\L(G)}(h^{-1}(x^{\prime}),h^{-1}(y^{\prime}))=d_{G}(x^{\prime},y^{\prime})$.
\end{description}

This prove \eqref{eq:isometric1} and guarantees \eqref{eq:IneqQIso1} for $x,y \in PM_\L V(\L(G))$ when we take $x:= h(x^{\prime})$ and $y:= h(y^{\prime})$.
We know that there exist $X_1,X_2,Y_1,Y_2 \in PM_\L V(\L(G))$ with $x \in [X_1 , X_2]$ and $y \in [Y_1 , Y_2]$ such that $d_{\L(G)}(x,X_1)=\e_x$, $d_{\L(G)}(x,X_2)=\d_x$, $d_{\L(G)}(y,Y_1)=\e_y$, $d_{\L(G)}(y,Y_2)=\d_y$ and $[X_1X_2]$, $[Y_1Y_2]$ are two half-edges in $\L(G)$. Hence, we have $h(x) \in [h(X_1)h(X_2)]$, $h(y) \in [h(Y_1)h(Y_2)]$ with $d_{G}(h(x),h(X_1)) = \e_x$, $d_{G}(h(x),h(X_2))=\d_x$, $d_{G}(h(y),h(Y_1))=\e_y$ and $d_{G}(h(y),h(Y_2))=\d_y$; besides $[h(X_1)h(X_2)]$ and $[h(Y_1)h(Y_2)]$ are two half-edges in $G$.

Note that if $[X_1X_2] = [Y_1Y_2]$, then $d_{G}(h(x) , h(y)) = d_{\L(G)}(x , y)$. Otherwise, we have

\begin{equation}\label{eq:DistLine}
d_{\L(G)}(x , y) = \min \left\{
                \begin{array}{ll}
                  d_{\L(G)}(X_1 , Y_1) + \e_x + \e_y, \\
                  d_{\L(G)}(X_1 , Y_2) + \e_x + \d_y, \\
                  d_{\L(G)}(X_2 , Y_1) + \d_x + \e_y, \\
                  d_{\L(G)}(X_2 , Y_2) + \d_x + \d_y
                \end{array}
              \right\}
\end{equation}

and

\begin{equation}\label{eq:DistG}
d_{G}(h(x) , h(y)) = \min \left\{
                \begin{array}{ll}
                  d_{G}(h(X_1) , h(Y_1)) + \e_x + \e_y, \\
                  d_{G}(h(X_1) , h(Y_2)) + \e_x + \d_y, \\
                  d_{G}(h(X_2) , h(Y_1)) + \d_x + \e_y, \\
                  d_{G}(h(X_2) , h(Y_2)) + \d_x + \d_y
                \end{array}
              \right\}.
\end{equation}

Let us consider $X_i,Y_j$ with $i,j\in\{1,2\}$, $\a \in \{\e_x,\d_x\}$ and $\b \in \{\e_y,\d_y\}$ such that $d_{\L(G)}(x , y) = d_{\L(G)}(X_i , Y_j) + \a + \b$. Hence, by \eqref{eq:isometric1} we have

$$
\begin{aligned}
d_{\L(G)}(x , y) &= d_{\L(G)}(X_i , Y_j) + \a + \b,\\
&\geq d_{G}(h(X_i) , h(Y_j)) + \a + \b, \\
&\geq d_{G}(h(x) , h(y)).
\end{aligned}
$$
\end{proof}

%
%
%

The following result is a consequence of Proposition \ref{p:Lipschitz}.

\begin{remark}\label{r:hVertex}
Let $x$ and $y$ be in $V(\L(G))$, then we have that
\begin{equation*}
    d_{\L(G)}(x , y) = d_{G}(h(x) , h(y)).
\end{equation*}
\end{remark}

\medskip

We also have a kind of reciprocal of Proposition \ref{p:Lipschitz}.

\begin{lemma}\label{t:lineQIso}
For every $x,y \in \L(G)$ we have
\begin{equation}\label{eq:IneqQIso2}
    d_{\L(G)}(x , y) \le  d_{G}(h(x) , h(y)) + 2 l_{max},
\end{equation}
where $l_{max}:= \sup_{e\in E(G)} L(e)$.
\end{lemma}

\begin{proof}
First of all, we prove \eqref{eq:IneqQIso2} for $x,y \in PM_\L V(\L(G))$. In order to prove it, we can assume that $\diam_{\L(G)}h^{-1}(h(x)), \diam_{\L(G)}h^{-1}(h(y)) > 0$ (i.e., $h(x),h(y) \in V(G)$ and $\deg_G(h(x)), \deg_G(h(x)) > 2$), since otherwise the argument is easier. Thus, by definition of $\L(G)$ we have a complete subgraph $K_{\deg(v)}$ associated to $v \in V(G)$ and $h^{-1}(v)=PM_\L(\L(G))\cap K_{\deg(v)}$.
Let us choose $x^{\prime\prime} \in h^{-1}(h(x))$, $y^{\prime\prime} \in h^{-1}(h(y))$ with $d_{\L(G)}(x^{\prime\prime} , y^{\prime\prime}) = d_{\L(G)}(h^{-1}(h(x)) , h^{-1}(h(y)))$. Consider a geodesic $\g$ joining $x^{\prime\prime}$ and $y^{\prime\prime}$ in $\L(G)$. Let $V_1$ (respectively, $V_2$) be the closest vertex to $x^{\prime\prime}$ (respectively, $y^{\prime\prime}$) in $\g$. It is easy to check that, since $h^{-1}(v) = PM_\L(\L(G)) \cap K_{\deg(v)}$ and $L([V_{e_i},V_{e_j}]) = (L(e_i) + L(e_j)) /2$, we have

\[
\displaystyle d_{\L(G)}(V_1 , x) \le d_{\L(G)}(V_1 , x^{\prime\prime}) + \sup_{e\in E(G)} L(e),
\]

\[
\displaystyle d_{\L(G)}(V_2 , y) \le d_{\L(G)}(V_2 , y^{\prime\prime}) + \sup_{e\in E(G)} L(e),
\]
and since

\[
d_{\L(G)}(x^{\prime\prime} , V_1) + d_{\L(G)}(V_1 , V_2) + d_{\L(G)}(V_2 , y^{\prime\prime}) = d_{\L(G)}(x^{\prime\prime} , y^{\prime\prime}),
\]
we deduce \eqref{eq:IneqQIso2} for $x,y \in PM_\L V(\L(G))$.

Now, let us consider $X_{i^\prime},Y_{j^\prime}$ with $i^\prime,j^\prime \in\{1,2\}$, $\a^\prime \in \{\e_x,\d_x\}$ and $\b^\prime \in \{\e_y,\d_y\}$ such that $d_{G}(h(x) , h(y)) = d_{G}(h(X_{i^\prime}) , h(Y_{j^\prime})) + \a^\prime + \b^\prime$. Hence, we have

$$
\begin{aligned}
d_{G}(h(x) , h(y)) &= d_{G}(h(X_{i^\prime}) , h(Y_{j^\prime})) + \a^\prime + \b^\prime, \\
& \geq d_{\L(G)}(X_{i^\prime} , Y_{j^\prime}) - 2 l_{max} + \a^\prime + \b^\prime,
\end{aligned}
$$
finally, \eqref{eq:DistLine} gives the condition.
\end{proof}

\medskip

It is easy to see that $G \setminus h(\L(G))$ is the union of the half-edges of $G$ such that one of its vertices has degree $1$; thus the following fact holds.

\begin{remark}\label{r:QuasiIsom}
$h$ is a $(l_{max}/2)$-full $(1,2l_{max})$-quasi-isometry with $l_{max}=\sup_{e\in E(G)} L(e)$.
\end{remark}

\medskip

Now, let us consider a cycle $C$ in $G$. We define $g_C: C \longrightarrow \L(G)$ in the following way; $g_C(Pm(e)):=V_e$ for $e\in E(G)\cap C$;
if $C^*$ is the cycle in $\L(G)$ with vertices $\cup_{e\in E(G)} g_C(Pm(e))$, then one can check that
$h|_{C^*}: C^* \longrightarrow C$
is a bijection;
we define
\begin{equation}\label{eq:SimpleCycle}
g_C:= (h|_C^*)^{-1}: C \longrightarrow C^*.
\end{equation}

\begin{corollary}\label{c:TriangGeod}
Let $C$ be a geodesic polygon in a graph $G$ that is a cycle and let $g_C$ be the function
defined by \eqref{eq:SimpleCycle}. Then, $C^*:=g_C(C)$ is a geodesic polygon in $\L(G)$ with the same number of edges than $C$.

Furthermore, if $\g$ is a geodesic in $C$, then $g_C(\g)$ is a geodesic in $\L(G)$ with $L(g_C(\g))=L(\g)$.
\end{corollary}

\begin{proof}
First of all, note that $L(C)=L(C^*)$ since if $E(C)=\{e_1,\ldots,e_n\}$ with $e_1 \cap e_n \neq \varnothing$ and $e_i \cap e_{i+1} \neq \varnothing$ for $1 \le i < n$, then $L(C)=\displaystyle\sum_{i=1}^{n} L(e_i)$ and $L(C^*)=L(e_1)/2 + \displaystyle\sum_{i=1}^{n-1} (L(e_i) + L(e_{i+1}))/2 + L(e_n)/2$.

Now, let us consider a geodesic $\g$ in $C$ joining $x$ and $y$. Since $g_C(\g)$ is a path joining $g_C(x)$ and $g_C(y)$, we have that $d_{\L(G)}(g_C(x) , g_C(y)) \le d_{C^*}(g_C(x) , g_C(y)) = d_{G}(x , y)$;
Proposition \ref{p:Lipschitz} gives $d_{\L(G)}(g_C(x) , g_C(y)) \geq d_{G}(h(g_C(x)) , h(g_C(y)))=d_{G}(x , y)$. Then we obtain that

\[
d_{\L(G)}(g_C(x) , g_C(y)) = d_{G}(x , y).
\]

Since $\g$ is an arbitrary geodesic in $C$ we obtain that $g_C$ maps geodesics in $G$ (contained in $C$) in geodesics in $\L(G)$ (contained in $C^*$).
\end{proof}

Now, we deal with the geodesics in $\L(G)$.

\begin{lemma}\label{l:GeodLine}
Let $\g^*$ be a geodesic joining $x$ and $y$ in $\L(G)$. Then $h(\g^*)$ is a path in $G$ joining $h(x)$ and $h(y)$, which is the union of three geodesics $\g_1$, $\g_2$, $\g_3$ in $G$, with $h(x) \in \g_1$, $h(y) \in \g_3$ and $0 \le L(\g_1),L(\g_3) < \sup_{e\in E(G)} L(e)$.
\end{lemma}

\begin{proof}
Note that if $x,y$ are contained in one edge $[V_1,V_2]$ of $\L(G)$, then $\g^* \subset [V_1,V_2]$ and $h(\g^*)$ is a geodesic in $G$ joining $h(x)$ and $h(y)$, since $h(\g^*) \subset \g:= [h(V_1) h(Pm_\L([V_1,V_2]))]\cup[h(Pm_\L([V_1,V_2])) h(V_2)]$ and $\g$ is a geodesic in $G$ by Remark \ref{r:hVertex}.

If $x,y$ are not contained in the same edge of $\L(G)$, then let us consider $V_\a$ as the closest vertex in $\g^*$ to $\a$, for $\a \in \{x,y\}$ (it is possible to have $V_x=x$ or $V_y=y$). By Remark \ref{r:hVertex}, we have that $h([V_x V_y])=[h(V_x) h(V_y)]$ is a geodesic joining $h(V_x)$ and $h(V_y)$ in $G$; moreover, $h(\g^*) = [h(x) h(V_x)] \cup [h(V_x) h(V_y)] \cup [h(V_y) h(y)]$ where $[h(x) h(V_x)]$ and $[h(V_y) h(y)]$ are geodesics in $G$ since $x, V_x$ (respectively $y, V_y$) are contained in the same edge of $\L(G)$. This finishes the proof, since $L(e^*)\le \sup_{e\in E(G)} L(e)$ for every $e^* \in E(\L(G))$.
\end{proof}

The arguments in the proof of Lemma \ref{l:GeodLine} give the following result.

\begin{lemma}\label{l:GeodLinek}
Let $G$ be a graph with edges of length $k$ and $\g^*$ be a geodesic of $\L(G)$ joining $x$ and $y$ with $x,y \in PM_\L V(\L(G))$. Then $h(\g^*)$ is the union of three geodesics $\g_1^*$, $\g_2^*$, $\g_3^*$ in $G$ with $h(x) \in \g_1^*$, $h(y) \in \g_3^*$ and $0 \le L(\g_1^*),L(\g_3^*) \le k/2$.
\end{lemma}

Also, we shall need a property of geodesic quadrilaterals in $G$.

\begin{lemma}\label{l:DistBetween1}
   For every $x,y,u,v$ in the graph $G$, let us define $\G:=[xu] \cup [uv] \cup [vy]$. If $L([xu]), L([vy]) \le \sup_{e\in E(G)} L(e)$, then

   \begin{equation}\label{eq:LemDistBetween1}
     \forall\ \a \in \G \ , \ \exists\ \b \in [xy]  \ : \ d_{G}(\a,\b)  \le  2 \d(G) + \sup_{e\in E(G)} L(e).
   \end{equation}
\end{lemma}

\begin{proof}
 Let us consider the geodesic quadrilateral $Q=\{[xy] , [xu] , [uv] , [vy] \}$ and $\a \in \G$.
 If $\a \in [xu] \cup [vy]$, then there exists $\b \in \{x,y\} \subset [xy]$ such that $d_{G}(\a,\b)  \le  \max \{ L([xu]) , L([vy]) \} \le \sup_{e\in E(G)} L(e)$. If $\a \in [uv]$, then there exists $\a^\prime \in [xy] \cup [xu] \cup [vy]$ such that $d_{G}(\a , \a^\prime)  \le  2 \d(G)$ since $Q$ is a geodesic quadrilateral in $G$. So, there exists $\b \in [xy]$ such that $d_{G}(\a^\prime , \b)  \le  \sup_{e\in E(G)} L(e)$. Then, we obtain that $d_{G}(\a , \b) \le d_{G}(\a , \a^\prime) + d_{G}(\a^\prime , \b) \le  2 \d(G) + \sup_{e\in E(G)} L(e)$.
\end{proof}

\begin{theorem}\label{t:IneqLineGraph}
Let $G$ be a graph and consider $\L(G)$ the line graph of $G$. Then
\begin{equation}\label{eq:ThmIneqLineGraph}
    \d(G) \le \d(\L(G)) \le 5 \d(G) + 3 l_{max},
\end{equation}
with $l_{max}=\sup_{e\in E(G)} L(e)$. Furthermore, the first inequality is sharp: the equality is attained by every cycle graph.
\end{theorem}

\begin{proof}
First, let us consider a geodesic triangle $T=[xy]\cup[yz]\cup[zx]$ in $G$ that is a cycle. Hence, if $g_T$ is defined by \eqref{eq:SimpleCycle}, then Corollary \ref{c:TriangGeod} gives that $T^*=[g_T(x)g_T(y)]\cup[g_T(y)g_T(z)]\cup[g_T(z)g_T(x)]$ is a geodesic triangle in $\L(G)$; besides, by Proposition \ref{p:Lipschitz} we have that $d_{G}(u , v) \le d_{\L(G)}(g_T(u) , g_T(v))$ for every $u,v \in T$.

Let $\G = (\g_1 , \g_2 , \g_3)$ be a permutation of $( [xy] , [yz] , [zx] )$. So, by Proposition \ref{p:Lipschitz} we have

\[
\begin{aligned}
\displaystyle\sup_{a\in \g_1} \inf_{b\in \g_2 \cup \g_3} d_{G}(a , b)
& \le
\sup_{a\in \g_1} \inf_{b\in \g_2 \cup \g_3} d_{\L(G)} ( g_T(a) , g_T(b) )
\\
& \le
\sup_{a^*\in g_T(\g_1)} \inf_{b^*\in g_T(\g_2) \cup g_T(\g_3)} d_{\L(G)} ( a^* , b^* ).
\end{aligned}
\]

Since $\G$ is an arbitrary permutation, we obtain

\[
\d(T) \le \d(T^*) \le \d(\L(G)).
\]

This finishes the proof of the first inequality by Corollary \ref{c:ClosedCurve}.

\medskip

Now, let us consider a geodesic triangle $T^*=\{[x^* y^*] , [y^* z^*] , [z^* x^*]\}$ in $\L(G)$ that is a cycle, and a permutation $\G = (\g_1^* , \g_2^* , \g_3^*)$ of $( [x^*y^*] , [y^*z^*] ,$ $[z^*x^*] )$. So, by Lemma \ref{t:lineQIso} we have

\begin{equation}\label{eq:ThmIneqLineGraph1}
\begin{aligned}
\displaystyle\sup_{a^*\in \g_1^*} \inf_{b^*\in \g_2^* \cup \g_3^*} d_{\L(G)}(a^* , b^*)
& \le
\sup_{a^*\in \g_1^*} \inf_{b^*\in \g_2^* \cup \g_3^*} d_{G} ( h(a^*) , h(b^*) ) + 2 l_{max}
\\
& \le
\sup_{a\in h(\g_1^*)} \inf_{b\in h(\g_2^*) \cup h(\g_3^*)} d_{G} ( a , b ) + 2 l_{max},
\\
\displaystyle\sup_{a^*\in \g_1^*} d_{\L(G)}(a^* , \g_2^* \cup \g_3^*)
& \le
\sup_{a\in h(\g_1^*)} d_{G} ( a , h(\g_2^*) \cup h(\g_3^*) ) + 2 l_{max}.
\end{aligned}
\end{equation}

By Lemma \ref{l:GeodLine} we know that $h([x^* y^*])$ is the union of three geodesics $[\a_z^1 P_{\a_z^1}]$, $[P_{\a_z^1} P_{\a_z^2}]$ and $[P_{\a_z^2} \a_z^2]$ in $G$:

\[
h([x^* y^*]) = [\a_z^1 P_{\a_z^1}] \cup [P_{\a_z^1} P_{\a_z^2}] \cup [P_{\a_z^2} \a_z^2].
\]

Analogously, $h([y^* z^*])$ and $h([z^* x^*])$ are the union of three geodesics in $G$:

\[
h([y^* z^*]) = [\a_x^1 P_{\a_x^1}] \cup [P_{\a_x^1} P_{\a_x^2}] \cup [P_{\a_x^2} \a_x^2],
\]

\[
h([z^* x^*]) = [\a_y^1 P_{\a_y^1}] \cup [P_{\a_y^1} P_{\a_y^2}] \cup [P_{\a_y^2} \a_y^2].
\]

Now, let us consider a geodesic triangle $T:= \{[h(x^*) h(y^*)] , [h(y^*) h(z^*)] ,$ $[h(z^*) h(x^*)]\}$ in $G$. Without loss of generality we can assume that $\g_1^* = [x^*y^*]$, $\g_2^* = [y^*z^*]$ and $\g_3^* = [z^*x^*]$. Hence, by Lemma \ref{l:DistBetween1} we have that, if $\a \in h(\g_1^*)$ then there exists $\b \in [h(x^*) h(y^*)]$ such that

\begin{equation*}
    d_{G} ( \a , \b) \le 2 \d(G) + l_{max}.
\end{equation*}

Since $\b \in [h(x^*) h(y^*)]$, there exists $\b^\prime \in [h(y^*) h(z^*)] \cup [h(z^*) h(x^*)]$ such that

\begin{equation*}
    d_{G} (\b , \b^\prime) \le \d(G).
\end{equation*}

Without loss of generality we can assume that $\b^\prime \in [h(y^*) h(z^*)]$. If we consider the geodesic quadrilateral $\{ [\a_x^1 \a_x^2] , [\a_x^1 P_{\a_x^1}] , [P_{\a_x^1} P_{\a_x^2}] , [P_{\a_x^2} \a_x^2] \}$, then there exists $\a^\prime \in h([y^*z^*])$ such that

\begin{equation*}\label{eq:ThmIneqLineGraph2}
    d_{G} (\b^\prime , \a^\prime) \le 2 \d(G).
\end{equation*}

Thus, since $d_{G} ( \a , h(\g_2^*) \cup h(\g_3^*) ) \le d_{G} ( \a , \b) + d_{G} ( \b , \b^\prime) + d_{G} ( \b^\prime , \a^\prime)$ we obtain that

\begin{equation}\label{eq:ThmIneqLineGraph2}
    d_{G} ( \a , h(\g_2^*) \cup h(\g_3^*) )  \le  5 \d(G) + l_{max}.
\end{equation}

Then, by \eqref{eq:ThmIneqLineGraph1} and \eqref{eq:ThmIneqLineGraph2} we obtain

\[
\displaystyle\sup_{a^*\in \g_1^*} d_{\L(G)}(a^* , \g_2^* \cup \g_3^*)
 \le 5 \d(G) + 3 l_{max}.
\]

Finally, since $\G$ is an arbitrary permutation of any triangle that is a cycle, Corollary \ref{c:ClosedCurve} gives

\[
\d(\L(G)) \le  5 \d(G) + 3 l_{max}.
\]

Corollary \ref{c:examplesdelta} gives that $\d(G) = \d(\L(G)) =L(G)/4$ for every cycle graph $G$.
\end{proof}

\begin{remark}\label{t:examplesLine}
The cycle graphs are not the only graphs $G$ with $\d(\L(G))=\d(G)$, as the
following example shows.
Let $C_n$ be the cycle graph with $n$ vertices and every edge with length $k$, and
$u,v\in V(C_n)$ with $d_{C_n}(u,v)=2k$;
if $G$ is the graph obtained by adding the edge $[u,v]$ (also with length $k$) to $C_n$,
one can check that
$\d(\L(G))=\d(G)=kn/4$.
\end{remark}

Let us consider now graphs with edges of length $k$. We will improve Theorem \ref{t:IneqLineGraph} in this case.

\begin{corollary}\label{c:IneqLineGraphk}
Let $G$ be any graph such that every edge has length $k$ and consider $\L(G)$ the line graph of $G$. Then
\begin{equation*}
    \d(G) \le \d(\L(G)) \le 5 \d(G) + \frac{5k}2.
\end{equation*}
\end{corollary}

\begin{proof}
We just need to prove the second inequality. By Theorem \ref{t:DiscreCycle} it suffices to consider geodesic triangles in $\L(G)$ with vertices in $PM_\L V(\L(G))=PMV(\L(G))$.
Then the arguments in the proof of Theorem \ref{t:IneqLineGraph}, replacing Lemma \ref{l:GeodLine} by Lemma \ref{l:GeodLinek}, give the result.
\end{proof}

\medskip

In \cite[Corollary 20]{RSVV} we find the following result.

\begin{lemma}
\label{l:deg}
Let $G$ be any graph with $m$ edges such that every edge has length $k$. Then $\d(G)\le km/4$, and the equality is attained if and only if $G$ is a cycle graph.
\end{lemma}

\medskip

\begin{theorem}
\label{t:j}
Let $G$ be any graph such that every edge has length $k$, with $n$ vertices and maximum degree $\Delta$. Then
$$\d(\L(G))\le nk\Delta(\Delta-1)/8,$$
and the equality is attained if and only if $G$ is a cycle graph.
\end{theorem}

\begin{proof}
It is well known that $2(m(\L(G)) + m(G))= \sum_{i=1}^{n} (\deg_G(v_{i}))^{2}$,
where $\deg_G(v_{i})$ are the degrees of the vertices of $G$. Since
$2m(G)= \sum_{i=1}^{n} \deg_G(v_{i})$, Lemma \ref{l:deg} gives the inequality, and the equality is attained if and only if $G$ is a cycle graph.
\end{proof}

Using the argument in the proof of Theorem \ref{t:j} we also obtain the following inequality.

\begin{corollary}
\label{c:sumdd}
If $G$ is any graph such that every edge has length $k$, with $n$ vertices $v_1,\dots,v_n$, then
$$
\d(\L(G)) + \d(G) \le \frac{k}8 \sum_{i=1}^{n} (\deg_G(v_{i}))^{2} ,
$$
and the equality is attained if and only if $G$ is a cycle graph.
\end{corollary}

\

\section{Acknowledgements}
We would like to thank the referee for a careful reading of the manuscript
and for some helpful suggestions.

This work was partly supported by the Spanish Ministry of Science and Innovation
through project MTM2009-07800 and a grant from CONACYT
(CONACYT-UAG I0110/62/10), M\'exico.

\

\end{document}